\documentclass[12pt,english,reqno]{amsart}
\usepackage{geometry}
\usepackage{amsmath,amssymb,amsthm,bbm,mathtools,comment}
\usepackage{amsfonts}
\usepackage[shortlabels]{enumitem}
\usepackage[pdftex,colorlinks,backref=page,citecolor=blue]{hyperref}
\usepackage[mathscr]{euscript}
\usepackage[usenames,dvipsnames]{color}
\usepackage{tikz,babel,adjustbox}
\usepackage[numbers]{natbib}
\usepackage{graphicx}
\usepackage{caption}
\usepackage{subcaption}
\usepackage{verbatim}
\usepackage{array}
\usepackage[frame,cmtip,arrow,matrix,line,graph,curve]{xy}
\usepackage{graphpap, pstricks}
\usepackage{etoolbox}
\usepackage{pifont}
\usepackage[noabbrev,capitalize]{cleveref}
\usepackage[final]{microtype}
\usepackage{cmtiup}

\hypersetup{pdfpagemode=UseNone,pdfstartview={XYZ null null 1.00}}
\usetikzlibrary{shapes.misc,calc,intersections,patterns,decorations.pathreplacing}
\usetikzlibrary{arrows,arrows.meta,shapes,positioning,decorations.markings}
\tikzstyle arrowstyle=[scale=1]

\allowdisplaybreaks

\makeatletter
\def\@settitle{\begin{center}%
		\bfseries\Large
		\@title
	\end{center}%
}
\patchcmd{\@setauthors}{\MakeUppercase}{\normalsize}{}{}
\makeatother

\theoremstyle{plain}
\newtheorem{theorem}{Theorem}[section]		
\newtheorem{lemma}[theorem]{Lemma}
\newtheorem{claim}[theorem]{Claim}

\theoremstyle{remark}

\newcommand{\beq}[1]{\begin{equation}\label{#1}}
	\newcommand{\enq}[0]{\end{equation}}
\def\Prob{\mathbb{P}}

\newcommand{\sd}[2]{#1\,\triangle\, #2}

\title{Antichain Codes}

\author{Benjamin Gunby}
\address{Department of Mathematics, Rutgers University, Piscataway, NJ 08854, USA}
\email{bg570@connect.rutgers.edu}

\author{Xiaoyu He}
\address{Department of Mathematics, Princeton University, Princeton, NJ 08540, USA}
\email{xiaoyuh@princeton.edu}

\author{Bhargav Narayanan}
\address{Department of Mathematics, Rutgers University, Piscataway, NJ 08854, USA}
\email{narayanan@math.rutgers.edu}

\author{Sam Spiro}
\address{Department of Mathematics, Rutgers University, Piscataway, NJ 08854, USA}
\email{sas703@scarletmail.rutgers.edu}

\begin{document}
	\maketitle
	\begin{abstract}
		A family of sets $A$ is said to be an antichain if $x\not\subset y$ for all distinct $x,y\in A$, and it is said to be a distance-$r$ code if every pair of distinct elements of $A$ has Hamming distance at least $r$. Here, we prove that if $A\subset 2^{[n]}$ is both an antichain and a distance-$(2r+1)$ code, then $|A| = O_r(2^n n^{-r-1/2})$. This result, which is best-possible up to the implied constant, is a purely combinatorial strengthening of a number of results in Littlewood--Offord theory; for example, our result gives a short combinatorial proof of H\'alasz's theorem, while all previously known proofs of this result are Fourier-analytic. 
	\end{abstract}
	\section{Introduction}
	In this paper, motivated by considerations from Littlewood--Offord theory, we study the intersection of two classical combinatorial problems in the hypercube, namely that of finding large antichains and that of finding large distance-$r$ codes. 
	
	A family of sets $A\subset 2^{[n]}$ is an \emph{antichain} if $x\not\subset y$ for any distinct $x,y\in A$. For example, the layer 
	\[\binom{[n]}{k}=\{x\subset [n]:|x|=k\}\] 
	is an antichain for all $0 \le k \le n$, and it is a classical result of Sperner~\cite{sperner1928satz} that every antichain in the hypercube $2^{[n]}$ has size at most $\binom{n}{\lfloor n/2\rfloor}$. There are a huge number of strengthenings and variants of Sperner's theorem; we refer the reader to~\cite{engel1997sperner} for more background. 
	
	A family of vectors $B \subset \{0,1\}^n$ is called a \textit{distance-$r$ code} if the Hamming distance between any pair of vectors in $B$ is at least $r$; identifying $\{0,1\}^n$ and $2^{[n]}$ in the natural way, we call a family $A\subset 2^{[n]}$ a \textit{distance-$r$ code} if the symmetric difference $x\,\triangle\, y$ of any two distinct $x,y\in A$ always has size at least $r$.  One of the central problems of coding theory is to find large distance-$r$ codes with various desirable properties, and the existence of such codes has many applications in both pure and applied problems. We refer the reader to~\cite{guruswami2012essential} for more on coding theory, and mention only the basic fact (as evidenced by BCH codes) that the largest possible distance-$r$ codes in $2^{[n]}$ have size $\Theta(2^n n^{-\lfloor(r-1)/2\rfloor})$.
	
	
	Here, we aim to answer the following natural question: how large can the cardinality of an \emph{antichain code} in $2^{[n]}$ be? After some thought, one finds that it is difficult to do much better than taking the intersection of a large code and a large antichain; our main result shows that such constructions are indeed optimal.
	\begin{theorem}\label{thm:main}
		For any fixed $r\in \mathbb{N}$, if $A\subset 2^{[n]}$ is both an antichain and a distance-$(2r+1)$ code, then
		\[|A| = O\left(2^n n^{-r-1/2}\right).\]
	\end{theorem}
	This result is best-possible up to multiplicative constants. Indeed, as mentioned above, we know that there exists a distance-$(2r+1)$ code $A\subset 2^{[n]}$ with $|A|=\Theta(2^nn^{-r})$.  By simple averaging, it is easy to  show that there exists some $x\subset[n]$ for which $A \,\triangle\, x=\{a\,\triangle\, x:a\in A\}$, which is also a distance-$(2r+1)$ code, intersects the middle layer $\binom{[n]}{\lfloor n/2\rfloor}$ in at least $\Theta(n^{-r} \binom{n}{\lfloor n/2\rfloor})=\Theta(2^n n^{-r-1/2})$ sets; then $( A\,\triangle\, x)\cap \binom{[n]}{\lfloor n/2\rfloor}$ gives us an antichain code of the desired size.

	
	As mentioned earlier, the primary motivation for Theorem~\ref{thm:main} comes from the Littlewood--Offord theory of anticoncentration. In particular, Theorem~\ref{thm:main} may be viewed as a purely combinatorial abstraction of an important result of Hal\'asz~\cite{halasz1977estimates} that is widely used in the study of random matrices; see~\cite{ferber2021counting,kahn1995probability} and the many references therein. To explain this connection, we need to fill in some background, a task to which we now turn.
	
	Recall that the Littlewood--Offord problem asks the following: given a vector $\mathbf{a}=(a_1,\ldots,a_n)$ of non-zero real numbers, estimate 
	\[\rho(\mathbf{a})=\max_{\alpha\in \mathbb{R}} \Prob[\epsilon_1 a_1+\cdots +\epsilon_n a_n = \alpha],\]
	where the $\epsilon_i$'s are independent Bernoulli random variables with $\Prob[\epsilon_i=0]=\Prob[\epsilon_i=1]$. 
	In their study of random polynomials, Littlewood and Offord~\cite{LO} showed that $\rho(\mathbf{a})= O(n^{-1/2} \log n)$ for any such $\mathbf{a}$, and soon after, Erd\H{o}s~\cite{erdos1945lemma} used Sperner’s theorem to give a simple combinatorial proof of the sharp estimate $\rho(\mathbf{a}) \le 2^{-n}\binom{n}{n/2} = O(n^{-1/2})$.
	
	There has since been considerable interest in establishing better bounds on $\rho(\mathbf{a})$ under stronger assumptions on the arithmetic structure of $\mathbf{a}$. For example, Erd\H{o}s and Moser~\cite{erdos1947e736} proved that $\rho(\mathbf{a})=O(n^{-3/2}\log n)$ whenever all of the entries of $\mathbf{a}$ are distinct. S\'ark\"ozy and Szemer\'edi~\cite{sarkozy1965problem} improved this to $\rho(\mathbf{a})=O(n^{-3/2})$, which is asymptotically best possible, and subsequently, Stanley~\cite{stanley1980weyl} used algebraic arguments to establish sharp bounds for this problem. Of particular interest to us is a far-reaching generalisation of the S\'ark\"ozy--Szemer\'edi theorem due to Hal\'asz~\cite{halasz1977estimates}, one formulation of which is as follows.
	
	\begin{theorem}\label{thm:Halazs}
		Let $\mathbf{a}=(a_1,\ldots,a_n)$ be a vector of real numbers such that for any disjoint subsets $x,y\subset [n]$ with $|x|+|y|\le 2r$, we have
		\[\sum_{i\in x} a_i\ne \sum_{j\in y} a_j.\]
		Then
		\[\rho(\mathbf{a})=O_r\left(n^{-r-1/2}\right).\]
	\end{theorem}
	Note in particular that the $r=1$ case is equivalent to saying $a_i\ne a_j$ for any $i\ne j$, so the result in this case reduces to the S\'ark\"ozy--Szemer\'edi theorem on the Erd\H{o}s--Moser problem. 
	
	Hal\'asz's theorem has since become a widely used tool in the study of random matrices and random polynomials. All known proofs of Hal\'asz's theorem use Fourier analysis, and are very much arithmetic in nature. While searching for an analogue of Hal\'asz's theorem for some other anticoncentration problems (in discussions with Berkowitz), it became clear that it would be of some help to find a purely combinatorial proof of this result, in the spirit of Erd\H{o}s' classical arguments. Arguably, the primary motivation for Theorem~\ref{thm:main} is that Hal\'asz's theorem is an easy corollary.
	\begin{proof}[Proof that Thereom~\ref{thm:main} implies Theorem~\ref{thm:Halazs}.]
		First, we we may assume without loss of generality that $a_i>0$ for all $i$; this follows from noting that $\rho(\mathbf{a})$ is unaltered if we replace the $0/1$-valued Bernoulli random variables in its definition with $-1/1$-valued Rademacher random variables. 
		
		Next, fix any real number $\alpha$ and let $A$ denote the family of subsets $x\subset [n]$ such that $\sum_{i\in x} a_i=\alpha$.  Note that $A$ is an antichain, as having $x\subset y$ both in $A$ would imply
		\[\alpha=\sum_{i\in x} a_i<\sum_{i\in y} a_i=\alpha,\]
		with the second inequality using $a_i>0$ for all $i$.  We also claim that $A$ is a distance-$(2r+1)$ code.  Indeed, for any distinct $x,y\in A$ we must have
		\[\sum_{i\in x\setminus y} a_i=\sum_{j\in y\setminus x} a_j;\] this implies, by the hypothesis on $\mathbf{a}$, that $|x\,\triangle \,y|=|x\setminus y|+|y\setminus x|>2r$, so $A$ is a distance-$(2r+1)$ code.  Thus, the bound of Theorem~\ref{thm:main} applies to $A$, and this is equivalent to saying that the probability that $\epsilon_1 a_1+\cdots+\epsilon_n a_n=\alpha$ is $O_r(n^{-r-1/2})$ for any $\alpha$; this gives the desired bound on $\rho(\mathbf{a})$.
	\end{proof}
	
	

	\section{Proof of the main result}
	We start with some brief comments on notation. We adopt the convention that lower case letters $a,b,c,x,y,z$ represent subsets of $[n]$, and capital letters $A,S$ represent families of sets, i.e., subsets of $2^{[n]}$. If $S\subset 2^{[n]}$ and $r\ge 1$ is an integer, we write $\partial^r S$ for the $r$-fold shadow of $S$, i.e., the collection of sets which can be obtained by deleting $r$ elements of $[n]$ from some set in $S$.  Similarly, we write $\partial^{-r} S$ for the collection of sets which can be obtained by adding $r$ elements to some set in $S$. For singletons we abuse notation by writing $\partial^{r} x$ for $\partial^{r} \{x\}$ and similarly $\partial^{-r}x$ for $\partial^{-r}\{x\}$.  Finally, we write $x\,\triangle\, y$ for the symmetric difference of $x$ and $y$. Throughout, we omit floors and ceilings when they are not crucial.
	
	Before we state and prove the main lemma that drives the proof of Theorem~\ref{thm:main}, we recall one proof of Sperner's theorem that serves as our inspiration. The local-LYM inequality (see~\cite{bbbook}) asserts that for any $S \subset \binom{[n]}{k}$, we have
	\[|\partial S| \binom{n}{k-1}^{-1} \ge |S| \binom{n}{k}^{-1}.\]
	It is not hard to show using local-LYM that any antichain $A \subset 2^{[n]}$ may be `shifted', by means of taking shadows, into the middle layer without decreasing the size of the resulting family, whence we conclude that $|A| \le \binom{n}{\lfloor n/2\rfloor}$. 
	
	A natural approach to proving Theorem~\ref{thm:main}, say for antichain codes of distance 3 to be concrete, is to proceed along similar lines as above, except using the distance condition instead of the local-LYM inequality to generate more `local expansion'. Concretely, given $A \subset 2^{[n]}$ that is both an antichain and a distance-3 code, it is easy to see for all $k$ that $|\partial A_k| \ge k |A_k|$, where $A_k = A \cap \binom{[n]}{k}$; in particular, for $k \approx n/2$, this tells us that $|\partial A_k| \gtrsim n|A_k|/2$, which is a significant improvement over the rather modest bound $|\partial A_k| \gtrsim |A_k|$ promised by local-LYM. It is then natural to attempt to transform a given antichain code $A$ of distance 3, by means of taking shadows, into a family that lives in the middle layer that is about $n$ times larger, which would then show that $|A| \lesssim \binom{n}{\lfloor n/2\rfloor}/n = O(2^n n^{-3/2})$, as desired. 
	
	To implement such an idea, we need to deal with how the shadows of the different $A_k$'s overlap as we repeatedly take shadows to move $A$ into the middle layer. For example, an estimate of the following form would be ideal: for any $S \subset \binom{[n]}{k}$ disjoint from $A_k$ (i.e., thinking of $S$ as the shadow of all the $A_\ell$'s with $\ell > k$ in layer $k$), we have $|\partial (S \cup A_k)| \ge |S| + k|A_k|/100$. Unfortunately, this is too much to hope for: if there are no conditions on the arbitrary set $S$, then it can be arranged so that $\partial S$ contains the entirety of $\partial A_k$.  Nevertheless, the following lemma shows that something like this ideal estimate does in fact hold when one studies the expansion of antichain codes over (slightly) longer ranges.
	
	\begin{lemma}\label{lem:shadow}
		Let $n,r\ge 1$ be integers with $n\ge 8r$.  If $\frac{n}{2}+3r\le k\le \frac{3n}{4}$, $S\subset\binom{[n]}{k}$,
		and $A\subset\binom{[n]}{k-r}\backslash\partial^{r} S$ is a distance-$(2r+1)$
		code, then $|\partial^{3r}S\cup\partial^{2r}A|\ge|S|+\frac{n^r|A|}{4(2r)^{3r}}$.
	\end{lemma}
	
	\begin{proof}
		Suppose for the sake of contradiction that $|\partial^{3r}S\cup\partial^{2r}A|<|S|+\frac{n^r|A|}{4(2r)^{3r}}$.
		Observe that because $k\ge\frac{n}{2}+3r$, local-LYM tells us that $|\partial B|\ge|B|$ for
		any family inside one of the layers $\binom{[n]}{k},\ldots,\binom{[n]}{k-3r+1}$.		
		Pick a uniformly random $a\in A$.  Let $b$ and $b'$ denote a uniformly random pair of disjoint $r$-subsets of $a$, and $c$ a uniformly random $r$-subset of $[n]\setminus a$. Note that since $A$ is a distance-$(2r+1)$ code, $\sd{a}{b}$ is a uniform random element of $\partial^{r}A$, and similarly $\sd{a}{c}$ is a uniform random element of $\partial^{-r}A$. Also observe that $a$ is uniquely determined by specifying either $\sd{a}{b}$ or $\sd{a}{c}$.
		
		\begin{claim}
			The random element $\sd{a}{b}$ of $\partial^r A$ satisfies $\Prob[\sd{a}{b}\notin\partial^{2r}S]< \frac{1}{4}.$\label{Nclaim:down}
		\end{claim}
		\begin{proof}
			By assumption,
			\[
			|S|+\frac{n^r|A|}{4(2r)^{3r}}>|\partial^{3r}S\cup\partial^{2r}A|\ge|\partial^{2r}S\cup\partial^{r} A|=|\partial^{2r}S|+|\partial^{r} A\backslash\partial^{2r}S|\ge|S|+|\partial^{r} A\backslash\partial^{2r}S|.
			\]
			Since $A$ is a distance-$(2r+2)$ code in $\binom{[n]}{k-r}$ and $k-r\ge n/2$, we have
			\[|\partial^{r} A|={k-r\choose r}|A|\ge \left(\frac{k-r}{r}\right)^{r} |A|\ge  \frac{n^{r}}{(2r)^{r}}|A|\ge \frac{n^{r}}{(2r)^{3r}}|A|.\]
			Combining this with the inequality above gives
			\[|\partial^{r} A\backslash\partial^{2r}S|<\frac{n^{r}|A|}{4(2r)^{3r}}\le \frac{1}{4}|\partial^{r} A|.\]
			This proves the claim since $\sd{a}{b}$ is a uniform random element of $\partial^{r} A$.
		\end{proof}

		\begin{claim}
			The random element $\sd{a}{(b\cup b'\cup c)}$ of $\partial^{2r}\partial^{-r}A$
			satisfies $\Prob[\sd{a}{(b\cup b'\cup c)}\in\partial^{2r}S]<\frac{1}{8}$.\label{Nclaim:up-down-down}
		\end{claim}
		
		\begin{proof}
			Let $P$ be the set of pairs $(x,y)\in\partial^{-r}A\times\partial^{2r}S$
			for which $y\in\partial^{2r}x$. There are a total of $\binom{n-k+2r}{2r}|\partial^{2r}S|$
			ways of picking an element $y\in\partial^{2r}S$ and an element $x\in\partial^{-2r}y$,
			and of these ways at least $\binom{k}{2r}|S|$ satisfy $x\in S$.
			Since $A$ is disjoint from $\partial^{r} S$, we have that $\partial^{-r}A$ is disjoint
			from $S$, so
			\[
			|P|\le\binom{n-k+2r}{2r}|\partial^{2r}S|-\binom{k}{2r}|S|\le\binom{n/2}{2r}\cdot (|\partial^{2r}S|-|S|)< \frac{n^{3r}|A|}{4 \cdot 32^{r}r^{3r}(2r)!},
			\]
			where the second inequality used $k\ge\frac{n}{2}+2r$ and $n/2\ge 4r$, and the last inequality used $\binom{n/2}{2r}\le (n/2)^{2r}/(2r)!$ and $|\partial^{2r}S|\le|\partial^{3r}S\cup\partial^{2r}A|<|S|+\frac{n^{r}|A|}{4(2r)^{3r}}$. Also, observe that $\sd{a}{c}$ is a uniform random element of $\partial^{-r}A$,
			so $\sd{a}{(b\cup b'\cup c)}\in \partial^{2r}S$ if and only if
			$(\sd{a}{c},\sd{a}{(b\cup b'\cup c)})\in P$. Note that $b,b',c$ are chosen
			uniformly at random out of \begin{equation}\binom{k-r}{r}\cdot \binom{k-2r}{r}\cdot \binom{n-k+r}{r}\ge \left(\frac{k-r}{r}\right)^{r}\cdot \left(\frac{k-2r}{r}\right)^{r}\cdot \left(\frac{n-k+r}{r}\right)^{r}\ge \frac{n^{3r}}{16^{r}r^{3r}}\label{eq:binomial}\end{equation}
			possibilities, with the last step using $k-2r\ge n/2$ and $k\le 3n/4$.   Since at most $(2r)!$ tuples $(a,b,b',c)$ correspond to
			the same pair $(\sd{a}{c},\sd{a}{(b\cup b'\cup c)})$ (as $\sd{a}{c}$ determines $a$ and $c$, the only
			noninjectivity comes from swapping elements between $b,b'$), we obtain
			\[
			\Prob[\sd{a}{(b\cup b'\cup c)}\in\partial^{2r}S]\le\frac{(2r)!|P|}{|A|\cdot n^{3r}/16^{r}r^{3r}}<\frac{1}{4\cdot 2^{r}}\le \frac{1}{8}. \qedhere
			\]
		\end{proof}
		We use these two claims to show that $\partial^{2r}S$ has large distance-$2r$ `edge expansion'
		in $\binom{[n]}{k-2r}$.
		\begin{claim}
			There are at least \[\frac{n^{3r}|A|}{2\cdot 16^{r}r^{3r}}\] pairs $(x,y)\in\binom{[n]}{k-2r}^{2}$
			at distance $2r$ with $x\in\partial^{2r}S$ and $y\not\in\partial^{2r}S$.
		\end{claim}
		
		\begin{proof}
			By the two previous claims, $\Prob[\sd{a}{b}\notin\partial^{2r}S]<\frac{1}{4}$
			and $\Prob[\sd{a}{(b\cup b'\cup c)}\in\partial^{2r}S]<\frac{1}{8}$.
			Thus, if we generate a random pair $(x,y)=(\sd{a}{b},\sd{a}{(b\cup b'\cup c)})$,
			then $x,y$ automatically have distance $2r$, and with probability at least $1-\frac{1}{4}-\frac{1}{8}\ge \frac{1}{2}$ they satisfy $x\in\partial^{2r}S$ and $y\not\in\partial^{2r}S$. The
			total number of potential pairs $(x,y)$ as above is 
			\[
			|\partial^{r} A|\cdot\binom{k-2r}{r}\binom{n-k-r}{r}={k-r\choose d}|A|\cdot \binom{k-2r}{r}\binom{n-k-r}{r}\ge \frac{n^{3r}}{16^{r}r^{3r}}|A|,
			\]
			with this last inequality using \eqref{eq:binomial}; we get the desired result by multiplying by $1/2$.
		\end{proof}
		We finish by enumerating in two ways the set $Q$ of pairs $(y,z)\in\binom{[n]}{k-2r}\times\binom{[n]}{k-3r}$
		where $z\in\partial^{3r}S$, $z\in\partial^{r} y$, and $y\not\in\partial^{2r}S$.
		On the one hand, any pair $(x,y)$ from the previous claim corresponds
		to such a pair $(y,z)\in Q$ by taking $z=x\cap y$, and this correspondence
		is at most $\binom{k-2r}{r}$-to-one. By the previous claim, 
		\[
		|Q|\ge\frac{n^{3r}|A|}{2\cdot 16^{r}r^{3r}\cdot \binom{k-2r}{r}}\ge\binom{n}{r}^{-1}\frac{n^{3r}|A|}{2\cdot 16^{r}r^{3r}}.
		\]
		On the other hand, the number of ways to pick $z$ is $|\partial^{3r}S|$,
		the number of ways to pick $y\in\partial^{-r}z$ is $\binom{n-k+3r}{r}$, and of
		these pairs at least $\binom{k-2r}{r}|\partial^{2r}S|$ satisfy $y\in\partial^{2r}S$.
		Thus,
		\begin{align*}
			|Q|&\le \binom{n-k+3r}{r}|\partial^{3r}S|-\binom{k-2r}{r}|\partial^{2r}S|<\binom{n/2}{r}(|\partial^{3r}S|-|\partial^{2r}S|)\\&\le\binom{n/2}{r}(|\partial^{3r}S\cup\partial^{2r}A|-|S|)\le \binom{n/2}{r}\frac{n^{r}|A|}{4 (2r)^{3r}}.
		\end{align*}
		To get a contradiction, it suffices to show
		\[\binom{n/2}{r}\binom{n}{r}< 2\cdot 2^{-r}n^{2r},\]
		which follows from $\binom{n/2}{r}\le (n/2)^{r}$ and $\binom{n}{r}\le n^{r}$.
	\end{proof}
	
	Our main result follows quickly from Lemma~\ref{lem:shadow}.
	
	\begin{proof}[Proof of Theorem~\ref{thm:main}]
		Recall that our goal is to prove that if $A\subset 2^{[n]}$ is both an antichain and a distance-$(2r+1)$ code for some $r\ge 1$, then $|A|=O_r\left(2^n n^{-r-1/2} \right)$.
		
		Let $A_k =A\cap \binom{[n]}{k}$ and let $S_k\subset \binom{[n]}{k}$ consist of the sets $x\subset [n]$ which are contained in some element of $A$.  Since $A$ is an antichain, we have $A_{k+2r}\subset \binom{[n]}{k+2r}\setminus \partial^r S_{k+3r}$.  Since $S_{k}\subset \partial^{3r}S_{k+3r}\cup \partial^{2r}A_{k+2r}$, by Lemma~\ref{lem:shadow} we find for all $n/2\le k\le 3n/4-3r$ that
		\[|S_{k}|\ge |S_{k+3r}|+\frac{n^r|A_{k+2r}|}{4(2r)^{3r}}.\]
		By applying this result inductively, we find for all $k\ge n/2$ that
		\begin{equation}|S_k|\ge \frac{n^r}{4(2r)^{3r}} \sum_{\ell \in L_k} |A_\ell|,\label{eq:inductive}\end{equation}
		where $L_k$ is the set of those $k\le \ell\le 3n/4-3r$ such that $ \ell\equiv k+2r\mod 3r$.
		
		By decreasing the size of $A$ by at most a factor of 2, we may assume $A_\ell=\emptyset$ for all $k<n/2$.  Similarly by decreasing $A$ by a factor of at most $3r$, we may assume there exists some $i$ such that $A_\ell=\emptyset$ for all $\ell\not \equiv i\mod 3r$.  By standard estimates for the binomial coefficients, we have that $\sum_{\ell\ge 3n/4-3r} {n\choose \ell}\le 1.9^{n}$ for $n$ sufficiently large.  Thus if $k\ge n/2$ is the smallest integer with $k\equiv i-2r\mod 3r$, these assumptions and \eqref{eq:inductive} imply that for $n$ sufficiently large, we have
		\[\frac{n^r}{4(2r)^{3r}} \left(|A|-1.9^n\right)\le |S_k|\le {n\choose k}=O(2^{n}n^{-1/2}).\]
		Rearranging this gives us the desired bound on $|A|$.
	\end{proof}

	\section*{Acknowledgements}
	The second author was supported by the NSF Mathematical Sciences Postdoctoral Research Fellowships Program under Grant DMS-2103154, the third author was supported by NSF grants CCF-1814409 and DMS-1800521, and the fourth author was supported by the NSF Mathematical Sciences Postdoctoral Research Fellowships Program under Grant DMS-2202730. We are grateful to Noga Alon and Ross Berkowitz for stimulating conversations. 
	
	\bibliographystyle{amsplain}
	\bibliography{sample}
	
\end{document}